\documentclass[11pt, letterpaper]{amsart}

\usepackage[parfill]{parskip}
\usepackage{amssymb}
\usepackage{amsmath}
\usepackage{amsthm}
\usepackage{amsfonts}

\usepackage[T1]{fontenc}
\usepackage{palatino}
\usepackage[text={6.5in, 9in}, centering]{geometry}

\usepackage[usenames]{color}
\newtheorem{theorem}{Theorem}
\newtheorem{definition}[theorem]{Definition}
\newtheorem{lemma}[theorem]{Lemma}
\newtheorem{corollary}[theorem]{Corollary}

\newcommand{\C}{\mathbb{C}}
\newcommand{\R}{\mathbb{R}}

\newcommand{\N}{\mathbb{N}}

\begin{document}
\title{Homework}
\author{Zachary Bradshaw}
\begin{center}{\large \bf A geometric measure-type regularity criterion for solutions to the magnetohydrodynamical system.}
\medskip
\\{Z. Bradshaw}
\medskip
{\small \\Department of Mathematics
\\University of Virginia
\\P. O. Box 400137
\\Charlottesville, VA 22904-4137
\\zb8br@virginia.edu }
\end{center}

\medskip
{\bf Abstract:} Several formulations of a local geometric measure-type condition are imposed on super-level sets of mild solutions to the homogeneous incompressible 3D magnetohydrodynamical system with bounded initial data to prevent finite-time singularity formation. Supporting this, results regarding the existence, uniqueness, and real analyticity of mild solutions are established as is a sharp lower bound on the radius of analyticity.

\medskip
{\bf Keywords:} Mild Solutions; Magnetohydrodynamic equations; Regularity Criteria.
\medskip 

\section{Introduction}

In a very recent paper by Z. Gruji\'c (see [7]), a \emph{local geometric measure-type condition} is shown to prevent finite time singularity formation (with respect to the supremum norm) in 3D NSE with initial data in $L^\infty(\R^3)$.  The proof utilizes a relatively recent solution (due to Solynin [17]) to a generalization of the classical Beurling's problem which is concerned with estimating the harmonic measure at the origin and with respect to the unit disk of a closed subset of $[-1,1]$.  In the context of [7], sharp lower bounds on the uniform radius of spatial analyticity for mild solutions with initial data in $L^\infty (\R^3)$, along with a sparseness condition near the endpoint of a finite length regular time interval, $[0,T)$, reduce to a situation aptly estimated via Solynin's result and the harmonic measure maximum principle.  Resultantly, $T$ is not a singular time and, moreover, the solution can be extended beyond $T$.

In this work we adapt the methods of [7] to the case of 3D MHD to establish several regularity criteria involving analogous local geometric measure type conditions.  Central to these is a sparseness requirement (see Section 4 for the precise definition) which assumes a bound on the ratio between the (spatial, Lebesgue) measure of the region of intense behaviour on some interval centered at the point $x_0$ and the length of that interval. More specifically, for a vector field $F$ over $\R^n$, we assume for each point $x_0$ that there exists a unit vector, $d(x_0)$, and and a magnitude (less than a uniform radius of analyticity), $r(x_0)$, so that \[\frac {|\Omega_{F}\cap (x_0-rd,x_0+rd)|} {2r}\leq \delta,\]for some $\delta\in (0,1)$ where $\Omega_{F}$ is the set of points where $|F|$ is above some threshold.  

For 3D MHD, assumptions of sparseness are natural in the phenomenologically and numerically motivated theories of MHD turbulence due to the inherent anisotropy evident therein.  In particular, in the regime of strong turbulence, the magnetic and velocity fields undergo dynamic alignment and the regions of high turbulence populate coherent structures which are essentially two dimensional (c.f. [3], [4], [6]).  Thus, both the velocity field and the magnetic field should exhibit sparseness in a perpendicular direction to the (essentially) planar region where turbulence is high. 

Because the equations for MHD and NSE are formally similar, the mathematical theory of MHD is richly informed by that of NSE (c.f. [14] for the fundamentals).  Past work has been done to adapt key regularity results from NSE to MHD.  The utilization of coherence introduced in [5] and improved in [2] has been adapted in the case of ideal MHD in [18] and for non-ideal MHD in [10].  Mild solutions for MHD have been studied in the context of $BMO^{-1}$ (c.f. [12]) where the regularity result of Koch-Tataru are extended.  Real analyticity of (weak) solutions of 3D MHD has been treated for initial data in $H^1$ in [19].  This work is also notable as it establishes an analogue of the Beale-Kato-Majda result in the $L^\infty$ context (see [20] for the $BMO$ context).  To the knowledge of the author, necessary existence and regularity results have not previously been established for mild solutions with initial data in $L^\infty$ and are consequently included here.

This paper is organized as follows.  The subsequent two sections establish analogues to the features of mild solutions to NSE used in [7].  In section 2 we examine the existence and uniqueness of mild solutions to MHD and in section 3 we study spatial analyticity of these solutions.  The main tools used in these sections are approximation schemes adapted from [8], [9], [11].  In these sections we prefer terseness to redundancy with regard to the existing literature and omit all but the essential steps and those which must be augmented to suit our ultimate purposes.  Section 4 is dedicated to statements and proofs of the regularity results.

\section{Mild Solutions of MHD in $L^\infty$}
We will be interested in the 3D magnetohydrodynamical system (MHD) with bounded initial data.  For a fluid velocity $U$ and a magnetic field $B$ these equations read:
\begin{align*}U_t-\nu \triangle U +(U\cdot \nabla)U- (B\cdot \nabla)B +\nabla \Pi &=0
\\ B_t-\mu \triangle B + (U\cdot \nabla) B - (B\cdot \nabla )U &=0 
\\ \nabla\cdot U = \nabla\cdot B&= 0
\\ U(x,0)&=U_0(x)\in L^\infty(\R^D)
\\ B(x,0)&=B_0(x)\in L^\infty(\R^D),
\end{align*}where $\mu$ and $\nu$ are taken to be greater than zero and represent the magnetic diffusivity and kinematic viscosity respectively, and $\Pi=p+\frac 1 2 |B|^2$ is the total kinematic pressure.

Observe that for  $\lambda>0$ the fundamental solution of $\partial_t f-\lambda\triangle f=0$ is given (for $t>0$) by: \[G_\lambda(x,t)=\frac 1 {(2\pi \lambda t)^{D/2}}e^{\frac {-|x|^2}{4\lambda t}}=G(x,\lambda t),\]where $G$ is the fundamental solution of the heat equation.  

\begin{definition}Let $U_0,B_0\in L^\infty (\R^D)$, both divergence free.  The functions $U, B\in C((0,T),L^\infty(\R^D))$ comprise a mild solution to MHD on the time interval $[0,T)$ if, for every $(x,t)\in \R^D\times(0,T)$, they satisfy: 
\begin{align*}U_k(x,t)=&\int_{\R^D} G_\nu(x-w,t)U_{0k}(w)dw
\\&-\int_0^t\int_{\R^D}\partial_j G_\nu (x-w,t-s)U_j(w,s)U_k(w,s)dwds
\\&+\int_0^t\int_{\R^D}\partial_j G_\nu(x-w,t-s)B_j(w,s)B_k(w,s)dwds
\\&-\int_0^t\int_{\R^D}\nabla G_\nu (x-w,t-s)\Pi(w,s) dwds,
\\B_k(x,t)=&\int_{\R^D} G_\mu(x-w,t)B_{0k}(w)dw
\\&-\int_0^t\int_{\R^D}\partial_j G_\mu(x-w,t) U_j(w,s)B_k(w,s)dwds
\\&+\int_0^t\int_{\R^D}\partial_j G_\mu(x-w,t) B_j(w,s)U_k(w,s)dwds,
\end{align*}where
\begin{align*}\triangle \Pi=-\partial_j\partial_k (U_jU_k-B_jB_k).
\end{align*}
\end{definition}

We will arrive at a mild solution via the following approximation scheme:
\begin{align*}
U^{(0)}&=\Pi^{(0)}=B^{(0)}=0
\\ \partial_t U^{(n)}-\nu\triangle U^{(n)}&=-(U^{(n-1)}\cdot \nabla)U^{(n-1)}+(B^{(n-1)}\cdot \nabla) B^{(n-1)}- \nabla \Pi^{(n-1)}
\\ \partial_t B^{(n)}-\mu\triangle B^{(n)}&=-(U^{(n-1)}\cdot \nabla)B^{(n-1)}+(B^{(n-1)}\cdot \nabla)U^{(n-1)}
\\U^{(n)}(x,0)&=U_0(x)
\\B^{(n)}(x,0)&=B_0(x)
\\ \triangle \Pi^{(n)}&=-\partial_j\partial_k (U^{(n)}_j U^{(n)}_k - B^{(n)}_j B^{(n)}_k).
\end{align*}

In the following the existence of relevant solutions at each level is checked inductively.  To this end we need several lemmas.  Note that, up to the incorporation of a constant $\lambda$, these are identical to their counterparts in [11] and employ facts from [15, 16].  Proofs are consequently omitted.

\begin{lemma}If $f\in L^\infty $ then there exists a unique solution, $\pi\in BMO$, to the problem \[-\triangle  \pi=\partial_j\partial_k f.\]Furthermore, \[||\pi||_{BMO}\leq C ||f||_\infty.\]
\end{lemma}

\begin{lemma}For $T>0$ and $f_j\in L^\infty(0,T,BMO)$ where $j=1,\ldots,D$, we have the inequality\[\bigg|\int_0^t\int_{\R^D} \partial_j G_\lambda(x-y,t-s)f_j(y,s)dyds \bigg|\leq C_\lambda \sqrt T ||f||_{L^\infty (0,T,\mbox{BMO)}}.\]
\end{lemma}In the above $C_\lambda$ is a constant depending on $\lambda$ and $D$.

Hypothesize that $U^{(n-1)}$, $B^{(n-1)},$ and $\Pi^{(n-1)}$ all satisfy the scheme distributionally and that the first two functions are in $L^\infty$.  Then, $U^{(n)}$ and $B^{(n)}$ are obtained via Duhamel's procedure and $\Pi^{(n)}$ exists by Lemma 2.  It remains to show that $U^{(n)}$ and $B^{(n)}$ remain in $L^\infty$, a fact we present as a lemma.

\begin{lemma}Given $U_0,B_0\in L^\infty$, there exists $T_1>0$ so that if $T\leq T_1$ then for all $n\in \mathbb N$ and all pairs $(x,t)\in \R^D\times [0,T)$, the following bound is satisfied:
\begin{align*}||U^{(n)}||_{L^\infty([0,T)\times \R^D)}+||B^{(n)}||_{L^\infty([0,T)\times \R^D)}&\leq 2C_1 (||U_0||_{L^\infty([0,T)\times \R^D)}+||B_0||_{L^\infty([0,T)\times \R^D)}),
\end{align*}where $C_1$ is independent of $n$ and depends on $\mu$, $\nu$, and $D$, and \[T_1=\frac 1 {4 (||U_0||_{L^\infty (\R^D)}+ ||B_0||_{L^\infty (\R^D)})^2}.\]
\end{lemma}
\begin{proof}Using the formulas for $U^{(n)}$ and $B^{(n)}$ obtained via Duhamel's procedure, an integration by parts yields:
\begin{align*}U^{(n)}(x,t)&=\int_{\R^D}G_\nu(x-w,t)U_{0}(x)dw
\\ &-\int_0^t\int_{\R^D} \partial_j G_\nu(x-w,t-s)U_j^{(n-1)}(w,s)U^{(n-1)}(w,s)dwds
\\ &+\int_0^t\int_{\R^D} \partial_j G_\nu(x-w,t-s)B_j^{(n-1)}(w,s)B^{(n-1)}(w,s)dwds
\\ &-\int_0^t\int_{\R^D} \nabla G_\nu(x-w,t-s)\Pi^{(n-1)}(w,s)dwds
\\B^{(n)}(x,t)&=\int_{\R^D}G_\mu(x-w,t)B_{0}(x)dw
\\ &-\int_0^t\int_{\R^D} \partial_j G_\mu(x-w,t-s)U_j^{(n-1)}(w,s)B^{(n-1)}(w,s)dwds
\\ &+\int_0^t\int_{\R^D} \partial_j G_\mu(x-w,t-s)B_j^{(n-1)}(w,s)U^{(n-1)}(w,s)dwds.
\end{align*}
Applying Lemma 3 to each summand above (and noting that the products of approximates from the scheme appearing in the integrands are in BMO in virtue of being bounded and that the BMO norms are bound by twice the supremum norms of these terms) grants that:
\begin{align*}|U^{(n)}| & \leq C_{\nu} ||U_0||_{L^\infty (\R^D)}
\\ &+C_{\nu}\sqrt{T} ||U^{(n-1)}||_{L^\infty((0,T)\times \R^D)}^2
\\ &+  C_{\nu}\sqrt{T} ||B^{(n-1)}||_{L^\infty((0,T)\times \R^D)}^2
\\ &+C_\nu \sqrt T ||\Pi^{(n-1)}||_{L^\infty ((0,T),BMO)},
\\|B^{(n)}|&\leq   C_{\mu}||B_0||_{L^\infty (\R^D)}
\\ &+2C_{\mu} \sqrt{T}||U^{(n-1)}||_{L^\infty((0,T)\times \R^D)}||B^{(n-1)}||_{L^\infty((0,T)\times \R^D)}.
\end{align*}

We have via Lemma 2 that \[||\Pi^{(n-1)}||_{L^\infty ((0,T),BMO)}\leq C \big(||U^{(n-1)}||_{L^\infty((0,T)\times \R^D)}^2+||B^{(n-1)}||_{L^\infty((0,T)\times \R^D)}^2\big),\]where $C$ is the constant appearing in Lemma 2.
Letting $C_1$ be appropriately large we obtain the bound:
\begin{align*}|U^{(n)}|+|B^{(n)}|&\leq C_1 ( ||U_0||_{L^\infty (\R^D)}+ ||B_0||_{L^\infty (\R^D)})
\\ &+ C_1\sqrt T  \big(||U^{(n-1)}||_{L^\infty((0,T)\times \R^D)}^2 
\\ & + 2 ||U^{(n-1)}||_{L^\infty((0,T)\times \R^D)}||B^{(n-1)}||_{L^\infty((0,T)\times \R^D)}
\\ & + ||B^{(n-1)}||_{L^\infty((0,T)\times \R^D)}^2\big),
\end{align*}
which factors to give a bound independent of $n$, $x,$ and $t$:
\begin{align*}|U^{(n)}|+|B^{(n)}|&\leq C_1 ( ||U_0||_{L^\infty (\R^D)}+ ||B_0||_{L^\infty (\R^D)})
\\ &+ C_1\sqrt T  \big(||U^{(n-1)}||_{L^\infty((0,T)\times \R^D)}+||B^{(n-1)}||_{L^\infty((0,T)\times \R^D)}\big)^2.
\end{align*}
Set \[T_1= \frac 1 {4 (||U_0||_{L^\infty (\R^D)}+ ||B_0||_{L^\infty (\R^D)})^2},\]
and assume $T\leq T_1$.  Then, for $n=1$ we have
\[|U^{(1)}|+|B^{(1)}|\leq 2C_1 (||U_0||_{L^\infty (\R^D)}+ ||B_0||_{L^\infty (\R^D)}).\]
Inductively we see that
\[|U^{(n)}|+|B^{(n)}|\leq 2C_1 (||U_0||_{L^\infty (\R^D)}+ ||B_0||_{L^\infty (\R^D)}).\]  Since this bound is appropriately uniform we conclude.
\end{proof}

Existence of mild solutions is now treated.

\begin{theorem}Given $U_0,B_0\in L^\infty (\R^D)$, there exists $T_2>0$ so that, for $T<T_2$, mild solutions to MHD exist on $[0,T]$ with \[T_2=\frac 1 {(4C_1C_2(||U_0||_{L^\infty (\R^D)}+ ||B_0||_{L^\infty (\R^D)}))^2},\]  
where $C_2$ is a constant depending on $\mu$, $\nu,$ and $D$.  Furthermore, mild solutions are unique.
\end{theorem}

\begin{proof}We assume that $U_0,B_0\in L^\infty (\R^D)$.  Our first step will be to appropriately restrict $T$ in order to establish the inequality: 
\begin{align*}||U^{(n+1)}-U^{(n)}||_{L^\infty((0,T)\times \R^D)} +||B^{(n+1)}-B^{(n)}||_{L^\infty((0,T)\times \R^D)}&\leq \alpha(||U^{(n)}-U^{(n-1)}||_{L^\infty((0,T)\times \R^D)} 
\\&+||B^{(n)}-B^{(n-1)}||_{L^\infty((0,T)\times \R^D)}),
\end{align*}for some $\alpha\in (0,1)$.

To achieve this we proceed essentially as in Lemma 4. Let $\mathcal U ^{(n+1)}=U^{(n+1)}-U^{(n)}$ and $\mathcal B^{(n+1)} = B^{(n+1)}-B^{(n)}$ and, by applying Duhamel's principle and an integration by parts, obtain formulas for $\mathcal U ^{(n+1)}$ and $\mathcal B ^{(n+1)}$.  Applying the same bounding techniques as in lemma 4 we obtain

\begin{align*}|\mathcal U ^{(n+1)}|+|\mathcal B^{(n+1)}|&\leq C\sqrt T ||\mathcal U^{(n)}||_{L^\infty((0,T)\times \R^D)}
\\&\cdot\bigg(||U^{(n)}||_{L^\infty((0,T)\times \R^D)}+||B^{(n)}||_{L^\infty((0,T)\times \R^D)}
\\&\mbox{~~~~~}+ ||U^{(n-1)}||_{L^\infty((0,T)\times \R^D)}+||B^{(n-1)}||_{L^\infty((0,T)\times \R^D)} \bigg)
\\&+ C\sqrt T ||\mathcal B^{(n)}||_{L^\infty((0,T)\times \R^D)}
\\&\cdot\bigg(||U^{(n)}||_{L^\infty((0,T)\times \R^D)}+||B^{(n)}||_{L^\infty((0,T)\times \R^D)}
\\&\mbox{~~~~~}+ ||U^{(n-1)}||_{L^\infty((0,T)\times \R^D)}+||B^{(n-1)}||_{L^\infty((0,T)\times \R^D)} \bigg)
\\&\leq C\sqrt T 4C_1 (||U_0||_{L^\infty((0,T)\times \R^D)}+||B_0||_{L^\infty((0,T)\times \R^D)})
\\&\cdot \big(||\mathcal U^{(n)}||_{L^\infty((0,T)\times \R^D)}+||\mathcal B^{(n)}||_{L^\infty((0,T)\times \R^D)}\big).
\end{align*}
Set $C_2=C$ as it appears above and \[T_2=\frac 1 {(4C_1C_2(||U_0||_{L^\infty (\R^3)}+ ||B_0||_{L^\infty (\R^3)}))^2}.\]  Assuming $0<T<T_2$, and therefore, $T=\alpha T_2$ where $\alpha\in (0,1)$, we see, 
\[||\mathcal U^{(n+1)}||_{L^\infty((0,T)\times \R^D)}+||\mathcal B^{(n+1)}||_{L^\infty((0,T)\times \R^D)}\leq \alpha \big(||\mathcal U^{(n)}||_{L^\infty((0,T)\times \R^D)}+||\mathcal B^{(n)}||_{L^\infty((0,T)\times \R^D)}\big). \]
Consequently we obtain a bound which diminishes as $n$ escapes: \[||\mathcal U^{(n)}||_{L^\infty((0,T)\times \R^D)}+||\mathcal B^{(n)}||_{L^\infty((0,T)\times \R^D)}\leq \alpha^{n-1}(||\mathcal U^{(1)}||_{L^\infty((0,T)\times \R^D)}+||\mathcal B^{(1)}||_{L^\infty((0,T)\times \R^D)}).\]

A simple convergence argument now applies.  Noting that $\alpha^{n-1}$ vanishes as $n\rightarrow \infty$, the above bound ensures that both  $\{U^{(n)}\}$ and $\{B^{(n)}\}$ are pointwise Cauchy and so we obtain pointwise limits $U$ and $B$.  The dominated convergence theorem (the dominator is the respective Gaussian kernel, scaled) then verifies that $U$ and $B$ are indeed mild solutions.  

Uniqueness follows from the fact that for two mild solutions $U$, $B$ and $\tilde U$, $\tilde B$, we have the bound \[||U-\tilde U||_\infty+||B-\tilde B||_\infty\leq (|| U-\tilde U||_\infty+||B-\tilde B||_\infty)(4C\sqrt {\tilde T} (||U||_\infty+||\tilde U||_\infty+||B||_\infty+||\tilde B||_\infty ).\]Then, taking $\tilde T$ small enough so that \[4C\sqrt {\tilde T} (||U||_\infty+||\tilde U||_\infty+||B||_\infty+||\tilde B||_\infty )<1,\]we obtain $U=\tilde U$ and $B=\tilde B$ on $[0,\tilde T]$.  This argument is iterated to obtain the conclusion on $[0,T]$.  

\end{proof}

The following corollary will be used later to minimize the geometric relevance of the magnetic field in our regularity result.

\begin{corollary}Suppose that $U$ and $B$ are the mild solutions obtained above to MHD on the time interval $[0,T)$ (so, $U_0,B_0\in L^\infty (\R^D)$ and $T<T_2$).  Then, there exists $T_3>0$ so that \[||B||_{L^\infty((0,T_3)\times \R^D)}\leq 2C_3||B_0||_{L^\infty (\R^D)},\]where $C_3$ is a constant depending on $D$ and $\mu$.
\end{corollary}
\begin{proof}In the proof of Lemma 4 we saw the bound:
\begin{align*}|B^{(n)}|&\leq   C_{\mu}||B_0||_{L^\infty (\R^D)}
+2C_{\mu} \sqrt{T}||U^{(n-1)}||_{L^\infty((0,T)\times \R^D)}||B^{(n-1)}||_{L^\infty((0,T)\times \R^D)}.
\end{align*}Set\[ T_3=\mbox{min}\bigg\{T,\frac 1 {16C_\mu^2||U||_{L^\infty((0,T)\times \R^D)}^2}\bigg\}.\]Then, for any $(t,x)\in  [0,T_3)\times\R^D$,
\begin{align*}|B^{(n)}(t,x)|&\leq   C_{\mu}||B_0||_{L^\infty (\R^D)}
+ \frac {||U^{(n-1)}||_{L^\infty((0,T_3)\times \R^D)}} {2||U||_{L^\infty((0,T_3)\times \R^D)}}||B^{(n-1)}||_{L^\infty((0,T_3)\times \R^D)}.
\end{align*}Taking limits and setting $C_3=C_\mu$ we see,
\[||B||_{L^\infty((0,T_3)\times \R^D)}\leq C_3 ||B_0||_{L^\infty (\R^D)}+\frac 1 2 ||B||_{L^\infty((0,T_3)\times \R^D)}.\]
\end{proof}

\section{Spatial Analyticity of Mild Solutions in $L^\infty$}

In this section we adapt the arguments of [8, 9] from the case of NSE in $D$ dimensions.  Essential to later work will be a uniform bound on the analytic extensions of mild solutions on a certain complex domain.  We restrict our attention to highlighting this result while appealing to the existing literature to fill in the any technical omissions.  

Recalling the approximation scheme for the real variable discussion, let $u^{(n)}+iv^{(n)}$, $b^{(n)}+i c^{(n)}$, and $\pi^{(n)}+i\rho^{(n)}$ be the analytic extensions of $U^{(n)}$, $B^{(n)}$ and $\Pi^{(n)}$ respectively.  Real analyticity of the $n$-th approximation is a consequence of real analyticity of solutions of the heat and Poisson equations.  Substituting these extensions into the scheme and subsequently isolating real and imaginary parts yields the system:
\begin{align*}
\partial_t u^{(n)}-\nu\triangle u^{(n)}&= -(u^{(n-1)}\cdot \nabla)u^{(n-1)}+(v^{(n-1)}\cdot \nabla)v^{(n-1)}\\& +(b^{(n-1)}\cdot \nabla)b^{(n-1)}-(c^{(n-1)}\cdot \nabla)c^{(n-1)}-\nabla \pi^{(n-1)}
\\ \partial_t v^{(n)}-\nu\triangle v^{(n)} &= -(u^{(n-1)}\cdot \nabla)v^{(n-1)}-(v^{(n-1)}\cdot \nabla)u^{(n-1)}\\&+(b^{(n-1)}\cdot \nabla)c^{(n-1)}+(c^{(n-1)}\cdot \nabla)b^{(n-1)}-\nabla \rho^{(n-1)}
\\ \partial_t b^{(n)}-\mu\triangle b^{(n)}&=-(u^{(n-1)}\cdot \nabla)b^{(n-1)}+(v^{(n-1)}\cdot \nabla)c^{(n-1)}\\&+(b^{(n-1)}\cdot \nabla)u^{(n-1)}-(c^{(n-1)}\cdot \nabla)v^{(n-1)}
\\ \partial_t c^{(n)}-\mu\triangle c^{(n)}&=-(v^{(n-1)}\cdot \nabla)b^{(n-1)}-(u^{(n-1)}\cdot \nabla)c^{(n-1)}\\&(b^{(n-1)}\cdot \nabla)v^{(n-1)}+(c^{(n-1)}\cdot \nabla)u^{(n-1)}
\\ \triangle \pi^{(n)}&=-\partial_j\partial_k(u_j^{(n)}u_k^{(n)}-v_j^{(n)}v_k^{(n)})+\partial_j\partial_k(b_j^{(n)}b_k^{(n)}-c_j^{(n)}c_k^{(n)})
\\ \triangle \rho^{(n)}&= -\partial_j\partial_k (u_j^{(n)}v_k^{(n)}+u_k^{(n)}v_j^{(n)}) + \partial_j\partial_k (b_j^{(n)}c_k^{(n)}+b_k^{(n)}c_j^{(n)})
\\u^{(n)}(x,y,0)&=u_0(x,y)
\\v^{(n)}(x,y,0)&=v_0(x,y)
\\b^{(n)}(x,y,0)&=u_0(x,y)
\\c^{(n)}(x,y,0)&=c_0(x,y).
\end{align*}

Letting $y=\alpha t$ with $\alpha\in \R^n$ and $t\geq 0$ will allow us to find a sharp lower bound on the analyticity radius depending on $t$.  Substituting these terms into the approximation scheme, omitting for brevity the terms associated with the total pressure which are just solutions of Poisson equations analogous to the real variable scheme, we obtain the relationships:
\begin{align*}\partial_t u^{(n)}-\nu \triangle u^{(n)}&= \alpha_j (\partial_jv^{(n)})-(u^{(n-1)}\cdot \nabla)u^{(n-1)}+(v^{(n-1)}\cdot \nabla)v^{(n-1)}\\ &-(b^{(n-1)}\cdot \nabla)b^{(n-1)}+(c^{(n-1)}\cdot \nabla)c^{(n-1)}+\nabla \pi^{(n-1)}
\\ \partial_t v^{(n)}-\nu \triangle v^{(n)} &= -\alpha_j(\partial_j u^{(n)})-(u^{(n-1)}\cdot \nabla)v^{(n-1)}-(v^{(n-1)}\cdot \nabla)u^{(n-1)}\\&+(b^{(n-1)}\cdot \nabla)c^{(n-1)}+(c^{(n-1)}\cdot \nabla)b^{(n-1)}+\nabla \rho^{(n-1)}
\\ \partial_t b^{(n)}-\mu \triangle b^{(n)}&=-\alpha_j(\partial_j c^{(n)})-(u^{(n-1)}\cdot \nabla)b^{(n-1)}+(v^{(n-1)}\cdot \nabla)c^{(n-1)}\\&+(b^{(n-1)}\cdot \nabla)u^{(n-1)}-(c^{(n-1)}\cdot \nabla)v^{(n-1)}
\\ \partial_t c^{(n)}-\mu \triangle c^{(n)}&=-\alpha_j(\partial_j b^{(n)})-(v^{(n-1)}\cdot \nabla)b^{(n-1)}-(u^{(n-1)}\cdot \nabla)c^{(n-1)}\\&(b^{(n-1)}\cdot \nabla)v^{(n-1)}+(c^{(n-1)}\cdot \nabla)u^{(n-1)}
\\u^{(n)}(x,0,0)&=U_0(x)
\\v^{(n)}(x,0,0)&=v_0(x,0)=0
\\b^{(n)}(x,0,0)&=B_0(x)
\\c^{(n)}(x,0,0)&=c_0(x,0)=0.
\end{align*}
We now present bounds on the sum of the ${L^\infty([0,T)\times \R^n)}$ norms of $u^{(n)}$, $v^{(n)}$, $b^{(n)}$, and $c^{(n)}$ for appropriate values of $T$ (denote by $F_n$ the set $\{u^{(n)},v^{(n)},b^{(n)},c^{(n)}\}$).  Applying Duhamel's Principle and a subsequent integration by parts yields formulas for each approximate in $F_n$.  These approximates can then each be bounded using the methods seen in the proof of Lemma 4.  Combining these bounds ultimately yields
\begin{align*}\sum_{f\in F_n}||f||_{L^\infty((0,T)\times \R^D)}&\leq  C(||U_0||_{L^\infty (\R^D)}+||B_0||_{L^\infty (\R^D)})
\\& +C|\alpha|\sqrt t \bigg( \sum_{f\in F_{n}}||f||_{L^\infty((0,T)\times \R^D)}\bigg)
\\& + C\sqrt T \bigg( \sum_{f\in F_{n-1}}||f||_{L^\infty((0,T)\times \R^D)}\bigg)^2.
\end{align*}
Let $C_4$ heretofore denote twice the constant appearing above.  Provided $C_4|\alpha|\sqrt t \leq 1$ (that is, $|y|\leq \frac 1 {2C_4} \sqrt t$), we get
\begin{align*}\sum_{f\in F_n}||f||_{L^\infty((0,T)\times \R^D)}&\leq C_4(||U_0||_{L^\infty (\R^D)}+||B_0||_{L^\infty (\R^D)})
\\& + C_4\sqrt T \bigg( \sum_{f\in F_{n-1}}||f||_{L^\infty((0,T)\times \R^D)}\bigg)^2.
\end{align*}
A consequence of the construction of $F_0$ is that
\begin{align*}\sum_{f\in F_1}||f||_{L^\infty((0,T)\times \R^D)}&\leq C_4(||U_0||_{L^\infty (\R^n)}+||B_0||_{L^\infty (\R^D)}).
\end{align*}Set \[T_4= \frac 1 {16c_4^4(||U_0||_{L^\infty (\R^D)}+||B_0||_{L^\infty (\R^D)})^2}.\] If $T\leq T_4$ we have, proceeding inductively,
\begin{align*}\sum_{f\in F_n}||f||_{L^\infty((0,T_4)\times \R^D)}&\leq 2C_4(||U_0||_{L^\infty (\R^D)}+||B_0||_{L^\infty (\R^D)}).
\end{align*}
Setting $\rho(t)=\frac {\sqrt t} {2C_4}$ we conclude then that for $t\in (0,T_4)$ the sequence of analytic extensions of $U^{(n)}+B^{(n)}$ is uniformly bounded over \[D_t = \bigg\{x+iy\in \mathbb C ^D:|y|\leq\rho(t)\bigg\}.\] 

We will eventually desire an improvement of the above bound which leaves $D_t$ unchanged.  By paying a price on the size of the time interval, we can scale the bound by a factor of $\beta$ for $\beta\in (1/2,1]$.  The price is to restrict the length of the time interval to be less than $T_\beta$ where \[T_\beta= \frac {(2\beta-1)^2} {\beta^4} T_4.\]Assuming such a restriction, for $n=1$ we have
\begin{align*}\sum_{f\in F_n}||f||_{L^\infty((0,T)\times \R^D)}&\leq C_4(||U_0||_{L^\infty (\R^D)}+||B_0||_{L^\infty (\R^D)})
\\&\leq 2\beta C_4(||U_0||_{L^\infty (\R^D)}+||B_0||_{L^\infty (\R^D)})
\end{align*}
and, for subsequent $n$, induction yields
\begin{align*}\sum_{f\in F_n}||f||_{L^\infty((0,T)\times \R^D)}&\leq 2\beta C_4(||U_0||_{L^\infty (\R^D)}+||B_0||_{L^\infty (\R^D)}).
\end{align*}
We summarize the preceding discussion in the following lemma.
\begin{lemma}Given $U_0$ and $B_0$ in $L^\infty(\R^D)$, there exists a universal constant $C_4$ so that, for any $n\in \N$ and $\beta\in (1/2,1]$, for the value\[T_\beta= \frac {(2\beta-1)^2} {\beta^4}  T_4,\]we have,
\[||u^{(n)}+iv^{(n)}||_{L^\infty(\Omega)}+||b^{(n)}+ic^{(n)}||_{L^\infty(\Omega)}\leq  2\beta C_4(||U_0||_{L^\infty( \R^D)}+||B_0||_{L^\infty( \R^D)}),\]
where \[\Omega=\bigg\{(x+iy,t)\in \C^D\times (0,T):|y|\leq \frac {\sqrt t} {2C_4}\bigg\},\]and $0<T\leq T_\beta$.
\end{lemma}

We now show that the analytic extensions of the approximations of the mild solutions converge and that their limits correspond to analytic extensions of the mild solutions themselves.  It turns out that these limits are also the unique mild solutions to the (spatially) analytic MHD equations, but we forgo a proof of this as it is methodically redundant to existing work (e.g. [9]). 

\begin{theorem}Let $U_0$ and $B_0$ be in $L^\infty(\R^D)$ and let $U$ and $B$ be the (unique) real variable solution to MHD on the time interval $[0,T^*)$ where $T^*<T_2$.  Then, for any $0<T<\mbox{min}\{T^*,T_4\}$, and for any $t\in [0,T)$, $U(\cdot,t)$ and $B(\cdot,t)$ have analytic extensions for which the domains of analyticity include $D_t$, denote these by $u+iv$ and $b+ic$.  Furthermore, for $\beta\in (1/2,1]$, there exists $T_\beta^*=\mbox{min}\{T^*,T_\beta\}$ so that the following bound holds \[||u+iv||_{L^\infty(\Omega)}+||b+ic||_{L^\infty(\Omega)}\leq 2\beta C_4(||U_0||_{L^\infty( \R^D)}+||B_0||_{L^\infty( \R^D)}),\]where \[\Omega=\bigg\{(x+iy,t)\in \C^D\times (0,T_\beta^*):|y|\leq \frac {\sqrt t} {2C_4}\bigg\}.\] 
\end{theorem}
\begin{proof}Because at each $t$ the approximating functions in the analytic scheme converge on a set containing an accumulation point, namely $\R^D$, Vitali's theorem grants that they converge to analytic functions $u+iv$ and $b+ic$.  Then, since these agree with $U$ and $B$ on $\R^D$, they constitute analytic extensions of $U$ and $B$.  The bound follows immediately from the bounds on the approximations.
\end{proof}


\section{Geometric-Measure Type Regularity Criteria}

We begin by defining the geometric measure criteria alluded to in the introduction.

\begin{definition}Let $x_0$ be a point in $\R^3$, $r>0$, $S$ an open subset of $\R^3$ and $\delta\in (0,1)$.

The set $S$ is {\em linearly $\delta$-sparse around $x_0$ at scale $r$ } if there exists a unit vector $d$ in $S^2$ such that \[\frac {|S \cap (x_0-rd,x_0+rd)|} {2r}\leq \delta.\]
\end{definition}
We will be interested in sparseness of super-level sets.  For a function $f(x,t)$, a time $t$, and a threshold $M$, a super-level set is defined to be \[\Omega_f(t,M)=\{x\in \R^D:|f(x,t)|>M\}.\]

There is a significant amount of freedom in choosing how to relate various parameters and achieve the desired regularity outcome.  We begin with a very simple case, Theorem 11, which most closely mirrors [7].  Here, the sparseness is imposed singly on the intersection of super-level sets of $U$ and $B$.  Consequently, the (local) direction in which $U$ and $B$ are sparse must agree.  This is reasonable in the context of the above discussion regarding MHD turbulence.  It is, however, formally restrictive and subsequent results are presented to reveal where added subtlety can be achieved.  The proof of this result will illustrate Gruji\'c's argument and expedite discussion of later results.

The remaining two theorems achieve the same result as Theorem 11 but under relaxed assumptions.    Theorem 12 assumes sparseness on each field but with greater independence than in Theorem 11.  A technical parameter specifies a relationship between the thresholds of the superlevel sets for $U$ and $B$ on which these sparseness assumptions are made.  Both the direction of sparseness and the scale are, however, independent. Theorem 13 exploits the linearity of the magnetic field in the magnetic field equation in order to eliminate this sparseness condition.  A cost is here paid by demanding the sup norm of the magnetic field is suitably bounded (in a fashion dependent on how much we improved the uniform bound) by a scaling of the sup norm of a single velocity profile.  

The only non-classical result from the theory of harmonic measures which is of interest to us is due to Solynin [17].  It is included for convenience.
\begin{theorem}{(\em Solynin [17])}.  Let $K$ be a closed subset of $[-1,1]$ such that $|K|=2\gamma$ for some $0<\gamma <1$.  Suppose further that $0\in \mathbb D\setminus K$.  Then,
\[\omega(0,\mathbb D, K)\geq\omega(0,\mathbb D, K_\gamma)=\frac 2 \pi \arcsin \frac {1-(1-\gamma)^2} {1+(1-\gamma)^2}\]
where $K_\gamma=[-1,-1+\gamma]\cup[1-\gamma,1]$.
\end{theorem}Other necessary results can be found in [1, 13] and are also listed in [7].

Note that in what follows $T_4=T_4(t_0)$ and $C_4=C_4(t_0)$ are determined in the context of Theorem 8 with initial data $U(t_0)$ and $B(t_0)$ for some time $t_0$.  These mild solutions are, up to a shift in the time variable, just the restrictions of the original solutions to the time interval $[t_0,t_0+T_4)$.  Also, we assume $C_4\geq 1$ and observe that, if this is not the case, we can re-determine constants and time interval lengths to reflect the constant $\max\{1,C_4\}$.

\begin{theorem} Suppose $U_0,B_0\in L^\infty$ and consider the corresponding mild solution comprised of $U$ and $B$ defined on an interval of regularity $[0,T)$.  Let $\delta\in (0,1)$, $h=h(\delta)=\frac 2 \pi \arcsin \frac {1-\delta^2} {1+\delta^2}$, $\alpha=\alpha(\delta)\geq \frac {1-h} {h}$ satisfying $\frac 1 2 \geq \frac 1 {2^{1/h}(2C_4)^\alpha}$.  
Assume there exists $\epsilon>0$ so that for any $t_0\in (T-\epsilon,T)$, either 
\begin{enumerate}
\item $t_0+T_4> T$, or,
\item there exists a time $t=t(t_0)\in [t_0+T_4/4,t_0+T_4]$  so that, for any $x_0\in \R^3$, with \[M=\frac 1 {2^{1/h}(2C_4)^\alpha}(||U(t_0)||_\infty +||B(t_0)||_\infty)\]
and\[S=\Omega_{U}(M,t)\cap \Omega_{B}(M,t),\]
there exists $r$ with $0<r<\rho(T_4/4)$ such that $S$ is linearly $\delta$-sparse around $x_0$ at scale $r$ in some direction $d$.
\end{enumerate}
Then, $T$ is not a singular time.
\end{theorem} 

It will be clear from the following proof that (2) needs only hold at finitely many times in $(T-\epsilon,T)$ provided these are suitably spaced.

\begin{proof}In the case of (1) we are done as the solutions with initial data take at $t_0$ are uniformly bounded on the interval $(t_0,t_0+T_4)$ which contains $T$.

In the case of (2) we apply an iterative argument which ultimately reduces to case (1).  Our main task is to establish that for any $t_0$ there exists a time $t(t_0)$ so that, for all $x_0\in\R^3$,
\[|U(x_0,t)|+|B(x_0,t)|\leq A\]
where $A=||U(t_0)||_\infty+||B(t_0)||_\infty$.  Consequently, the procedure can be repeated with $t$ replacing $t_0$ and, as each iteration moves the initial time closer to $T$ by a non-vanishing length, case (1) will eventually be achieved.  

To begin, let $t_0\in (T-\epsilon,T)$ and $x_0\in \R^3$ be fixed.  Let $t=t(t_0)$ be as in the theorem and, therefore, by the sparseness assumption, there exists a length $r<\rho(T_4/4)$ and a direction vector $d$ so that 
\[\frac {S\cap(x_0-rd,x_0+rd)|} {2r}\leq \delta.\]
Observing the the MHD system is rotationally and translationally invariant, let $Q$ denote the transformation (rotation and translation) taking $x_0$ to $0$ and directing $d$ to be parallel to the first coordinate vector, $e_1$.  Let $U_{x_0,Q}$ and $B_{x_0,Q}$ comprise a mild solution to the transformed MHD with initial data taken at $t_0$.

By Theorem 8, on $(0,T_4 )$, $U_{x_0,Q}$ and $B_{x_0,Q}$ have analytic extensions satisfying the uniform bound \[|u+iv|+|b+ic|\leq 2C_4A,\]where the bounded terms are the appropriate analytic extensions.  Focussing on the extension of the first spatial coordinate axis, hereafter called just the real axis, we see the domains of analyticity contain the disk centred at $0$ of radius $r$, $D_r$ because $r<\rho (T_4/4)\leq \rho(t)$.

Let $K$ be the complement in $[-r,r]$ of $Q$ applied to $S\cap (x_0-rd,x_0+rd)$.  If $0\in K$ then \[|U_{x_0,Q}(0,t)|\leq M\leq \frac 1 2 A,\] and, as the same bound holds on $|B_{x_0,Q}(0,t)|$, we are done.  If $0\notin K$ we turn to the theory of harmonic measures.  

To apply the harmonic measure maximum principle (c.f. [Ahl] pp.39) observe that, uniformly in $D_r$ we have,
\[|u+iv|\leq 2C_4 A,\]
while uniformly in $K$ we have,
\[|u+iv|\leq M,\] and, consequently,
\[|U_{x_0,Q}(0,t)|\leq \bigg(\frac A {2^{1/h}(2C_4)^\alpha}\bigg)^{\omega(0,D_{r},K)}\bigg(2 C_4A\bigg)^{1-\omega(0,D_{r},K)}.\]

The sparseness assumption entails that $|K|\geq 2r(1-\delta)$.  Letting $\frac 1 {r}K=\{z\in \C:rz\in K\}$, we obtain $|\frac 1 r K|\geq 2(1-\delta)$.  Applying Theorem 10 with $\gamma=1-\delta$ to a subset $K'\subset K$ where $|K'|=2(1-\delta)$ yields (noting harmonic measure increasing in $K$), \[\omega(0,D_1,\frac 1 r K)\geq \omega(0,D_1,\frac 1 r K')\geq \omega(0,D_1,K_{\gamma})=\frac 2 \pi \arcsin \frac {1-(\delta)^2} {1+(\delta)^2}=h.\]
Since harmonic measure is invariant under conformal mappings, it is invariant under the mapping $z\mapsto r z$.  The previous inequality then implies
\[\omega(0,D_{r},K)\geq h.\]
Combining our bounds (noting $C_4>1$ and $M\leq 2C_4A$) we see
\begin{align*}|U_{x_0,Q}(0,t)|&\leq \bigg(\frac A {2^{1/h}(2C_4)^\alpha}\bigg)^{h}
\bigg(2 C_4A\bigg)^{1-h}\leq \frac 1 2A. 
\end{align*}
And, undoing our transformation $Q$, \[|U(x_0,t)|\leq \frac 1 2 A.\]
Proceeding identically results in the same bound for $|B(x_0,t)|$ which gives the conclusion for $x_0$:\[|U(x_0,t)|+|B(x_0,t)|\leq ||U(t_0)||_\infty+||B(t_0)||_\infty.\]As our selection of $x_0$ was arbitrary this holds uniformly and the iterative argument outlined at the onset of the proof allows us to conclude that $T$ is not a singular time.
\end{proof}


\begin{theorem}  Suppose $U_0,B_0\in L^\infty$ and consider the corresponding mild solution comprised of $U$ and $B$ defined on an interval of regularity $[0,T)$.  Let $\delta\in (0,1)$, $h=h(\delta)=\frac 2 \pi \arcsin \frac {1-\delta^2} {1+\delta^2}$, $\alpha=\alpha(\delta)\geq \frac {1-h} {h}$, and $\gamma\in (0,1)$.

Assume there exists $\epsilon>0$ so that for any $t_0\in (T-\epsilon,T)$, either 
\begin{enumerate}
\item $t_0+T_4> T$, or,
\item there exists a time $t=t(t_0)\in [t_0+T_4/4,t_0+T_4]$  so that, for any $x_0\in \R^3$, there exist $r_U$ and $r_B$ so that following sparseness conditions are met: 
\begin{itemize}\item $\Omega_{U}(t,M_U)$ is linearly $\delta$-sparse around $x_0$ at scale $r_U$ where $0<r_U \leq \rho(T_4/4)$ and $M_U=\frac {\gamma} {(2C_4)^\alpha}(||U(t_0)||_\infty+||B(t_0)||_\infty)$, and,
\item $\Omega_{B}(t,M_B)$ is linearly $\delta$-sparse around $x_0$ at scale $r_B$ where $0<r_B \leq \rho(T_4/4)$ and $M_B=\frac {(1-\gamma^h)^{1/h}} {(2C_4)^\alpha}(||U(t_0)||_\infty+||B(t_0)||_\infty)$.
\end{itemize}
\end{enumerate}
Then, $T$ is not a singular time.
\end{theorem}

\begin{proof}The same iterative argument seen in the proof of Theorem 11 is applied.  To ensure it holds we obtain for all $x_0$ the bounds
\begin{align*}|U(x_0,t)|&\leq \gamma^h(||U(t_0)||_\infty +||B(t_0)||_\infty)\mbox{, and,}
\\|B(x_0,t)|&\leq (1-\gamma^h)(||U(t_0)||_\infty +||B(t_0)||_\infty).\end{align*}
 
This is shown by cases depending on the inclusion of $x_0$ in the relevant super-level sets.  In the case $x_0\in \Omega_{U}(t,M_U)$, we have, by an identical argument to that in the previous proof, that
\[|U_{x_0,Q}(0,t)|\leq \big(\frac \gamma {(2C_4)^\alpha}(||U(t_0)||_\infty+||B(t_0)||_\infty)\big)^{\omega(0,D_{r_U},K)}\big(2 C_4(||U(t_0)||_\infty +||B(t_0)||_\infty )\big)^{1-\omega(0,D_{r_U},K)}.\]
Since $1\geq\omega(0,D_{r_U},K_U)>h>0$ and $\gamma\in (0,1)$, $\gamma^h\geq \gamma^{\omega(0,D_{r_U},K_U)}$.  The desired bound then follows.  The case for $x_0\in \Omega_{B}(t,M_B)$ is identical up to labelling.  

In the case that $x_0\notin \Omega_{U}(t,M_U)$, we have \[|U(x_0,t)|\leq \frac {\gamma} {(2C_4)^\alpha} (||U(t_0)||_\infty +||B(t_0)||_\infty).\]Since $h,\gamma \in (0,1)$, $\alpha>0$, and $2C_4> 1$, we clearly have \[\frac {\gamma} {(2C_4)^\alpha} \leq \gamma^h,\]which yields the desired bound.

In the case that $x_0\notin \Omega_{B}(t,M_B)$, we similarly have \[|B(x_0,t)|\leq \frac {(1-\gamma^h)^{1/h}} {(2C_4)^\alpha} (||U(t_0)||_\infty +||B(t_0)||_\infty),\]and, observing that \[\frac {(1-\gamma^h)^{1/h}} {(2C_4)^\alpha} \leq (1-\gamma^h),\] we conclude.

So, the initially stated bounds are valid and, iterating, we conclude $T$ is not a singular time.
\end{proof}


\begin{theorem}  Suppose $U_0,B_0\in L^\infty$ and consider the corresponding mild solution comprised of $U$ and $B$ defined on an interval of regularity $[0,T)$.  Let $\delta\in (0,1)$, $h=h(\delta)=\frac 2 \pi \arcsin \frac {1-\delta^2} {1+\delta^2}$, $\alpha=\alpha(\delta)\geq \frac {1-h} {h}$.   Additionally, let $\beta\in (1/2,1)$ satisfy $\beta^{1-h}\geq 1/(2C_4)^\alpha$.

Assume there exists $\epsilon>0$ and a collection of times $t_0,t_1,\ldots,t_k\in (T-\epsilon,T)$, so that 
\begin{enumerate}
\item $t_k+T_4> T$,
\item $t_{i+1}\in [t_i+T_\beta(t_i)/4,t_i+T_\beta(t_i)]$, and,
\item the following two criteria are met:
\begin{itemize}\item for each $i$ and for any $x_0$ there exists $r$ so that $\Omega_{U}(t_{i+1},M)$ is linearly $\delta$-sparse around $x_0$ at scale $r$ where $0<r \leq \rho(T_\beta/4)$ and $M=\frac 1 {(2C_4)^\alpha}(||U(t_i)||_\infty+||B(t_i)||_\infty)$, and,
\item $2C_3 ||B(t_0)||_\infty \leq (1-\beta^{1-h})(||U(t_0)||_\infty+||B(t_0)||_\infty)$.
\end{itemize}
\end{enumerate}
Then, $T$ is not a singular time.

\end{theorem} 

In the above we specified a particular time at which to begin our iterative argument in order to not trivialize the condition on the magnetic field.  Note that in the context of the above theorem $\beta$ and $T_\beta$ reference Theorem 8. Also, when $C_4\geq 1$ we have $\beta^{1-h}\geq 1/(2C_4)^\alpha$.  

\begin{proof}We obtain for all $x_0$ the bounds
\begin{align*}|U(x_0,t_1)|&\leq \beta^h(||U(t_0)||_\infty +||B(t_0)||_\infty)\mbox{, and,}
\\|B(x_0,t_1)|&\leq (1-\beta^h)(||U(t_0)||_\infty +||B(t_0)||_\infty).\end{align*}
Applying the now familiar argument with the modification that the disk on which we are taking $U_{x_0,t_1}$ to be analytic is that on which the bound $|U_{x_0,t_1}|\leq 2\beta (||U(t_0)||_\infty +||B(t_0)||_\infty)$ holds (see Lemma 7), yields the desired bound on $U$:
\[|U(x_0,t_1)|\leq \beta^{1-h}(||U(t_0)||_\infty +||B(t_0)||_\infty)\]
A complementary estimate (noting $T_3\geq T_\beta$) on $|B(x_0,t_1)|$ follows from the second assumption and Corollary 6.  These grant that
\[|B(x_0,t_1)|\leq 2C_3 ||B(t_0)||_\infty\leq (1-\beta^{1-h})(||B(t_0)||_\infty + ||U(t_0)||_\infty).\]
Combining bounds,
\[|U(x_0,t_1)|+|B(x_0,t_1)|\leq ||U(t_0)||_\infty +||B(t_0)||_\infty.\]  Establishing identical relationships for $t_{i+1}$ and $t_i$ follows in the exact same manner and the iterative argument then grants that $T$ is not a singular time.

\end{proof}


\begin{thebibliography} {000000000}


\bibitem[1]{1}
L.V. Ahlfors, Conformal invariants: topics in geometric function theory, AMS Chelsea Pub., 2010.

\bibitem[2]{BeBe}
H. Beir\~ao Da Veiga, L. Berselli, On the regularizing effect of the vorticity direction in incompressible viscous flows, Differential Integral Equations,  15 No. 3, (2002) 345-356.

\bibitem[3]{DB} D. Biskamp, {\em Magnetohydrodynamic Turbulence}, Cambridge : Cambridge University Press, 2003.

\bibitem[4]{B05}
S. Boldyrev, On the spectrum of magnetohydrodynamic turbulence, Astrophys. J., L37 (2005) 626.

\bibitem[5]{CoFe}
P. Constantin, C. Fefferman, Direction of vorticity and the problem of global regularity for the Navier-Stokes equations, Indiana Univ. Math. J., Vol 42, No. 3 (1993) 775-789.

\bibitem[6]{GS95}
P. Goldreich and S. Sridhar, Toward a theory of interstellar turbulence. II. Strong Alfv\'enic turbulence,  Astrophys. J., 763 (1995) 438.

\bibitem[7]{Gr11}
Z. Gruji\'c, A geometric measure-type regularity criterion for solutions
to the 3D Navier-Stokes equations (submitted; arXiv:1111.0217).

\bibitem[8]{GrKu}
Z. Gruji\'c and I. Kukavica, Space analyticity for the Navier-Stokes and related equations with initial data in $L^p$, J. Func. Anal.  152, No. 2 (1998) 447-466.

\bibitem[9]{Gu10}
R. Guberovi\'c, Smoothness of Koch-Tataru solutions to the Navier-Stokes equations revisited, Discrete Contin. Dyn. Syst., Vol. 27, No. 1 (2010) 231-236.


\bibitem[10]{HeXe}
C. He, Z. Xin, On the regularity of weak solutions to the magnetohydrodynamic equations, J. Differential Equations 213 (2005) 235-254

\bibitem[11]{Ku03}
I. Kukavica, On local uniqueness of weak solutions of the Navier-Stokes system with bounded initial data, J. Differential Equations 194 (2003) 39-50.

\bibitem[12]{MiYuZh}
C. Miao, B. Yuan, B. Zhang, Well-posedness for the incompressible magneto-hydrodynamic system, Math. Methods Appl. Sci 30 (2007) 961-976

\bibitem[13]{Nev}
R. Nevanlinna, \textit{Analytic Functions}, Springer-Verlag, 1970 (Translated from the Second German Edition).

\bibitem[14]{ST}
M. Sermange and R. Temam, Some mathematical questions related to MHD equations, Comm. Pure Appl. Math. 635 (1983).

\bibitem[15]{St93}
E. Stein, \textit{Harmonic Analysis}, Princeton Univ. Press, Princeton, New Jersey (1993).

\bibitem[16]{St70}
E. Stein, \textit{Singular Integrals}, Princeton Univ. Press, Princeton, New Jersey (1970).


\bibitem[17]{Sol99}
A. Yu. Solynin, Ordering of sets, hyperbolic metrics, and harmonic measures, J. Math. Sci. 95 (1999) 2256-2266.


\bibitem[18]{Wu00}
J. Wu, Analytic results related to magneto-hydrodynamic turbulence, Phys. D 136 (2000) 353-372.

\bibitem[19]{Wu02}
J. Wu, Bounds and new approaches for the 3D MHD equations, J. Nonlinear Sci.  12 (2002) 395-413. 

\bibitem[20]{Yu06}
B. Yuan, On the blow-up criterion of smooth solutions to the MHD system in BMO space, Acta Math. Appl. Sin. Engl. Set. 22 No. 3 (2006) 413-418.


\end{thebibliography}
\end{document}